\documentclass[12pt,reqno]{amsart}

\usepackage[utf8]{inputenc}
\usepackage[T1]{fontenc}

\usepackage[usenames,dvipsnames]{color}

\usepackage{amsfonts, amsmath, amssymb, amscd, latexsym, amsthm}

\usepackage[english]{babel}

\usepackage{enumerate}
\usepackage{longtable}
\usepackage{multirow,longtable,array}
\usepackage{geometry}
\usepackage{float}
\usepackage{adjustbox}
\usepackage{rotating}
\usepackage{tabularx}
\usepackage{booktabs}
\usepackage{array}

\usepackage{tikz}
\usetikzlibrary{shapes,arrows}
\usepackage{xypic}
\usepackage{pifont}

\usepackage[pagebackref]{hyperref}
\renewcommand*{\backref}[1]{}
\newcolumntype{Y}{>{\centering\arraybackslash}X}
\renewcommand*{\backrefalt}[4]{%
	\ifcase #1
	(\tt not cited)%
	\or
	(\tt cited on page~#2)%
	\else
	(\tt cited on pages~#2)%
	\fi}
\usepackage{hyperref}
\hypersetup{
	colorlinks = true,
	urlcolor   = blue,
	citecolor  = blue
}

\newtheorem{theorem}{Theorem}[section]
\newtheorem{lemma}[theorem]{Lemma}

\newtheorem{corollary}[theorem]{Corollary}

\theoremstyle{remark}
\newtheorem{remark}[theorem]{\bf Remark}

\def\2r{\mathrm{rank}}

\def\vep{\varepsilon}
\def\NN{\mathbb{N}}

\def\QQ{\mathbb{Q}}

\def\ZZ{\mathbb{Z}}

\def\KK{\mathbb{K}}

\begin{document}
	\selectlanguage{english}
	\title[On the unit group   of some triquadratic   fields of 2-rank 2]{On the unit group   of some triquadratic   fields}

	\author[M. M. Chems-Eddin]{Mohamed Mahmoud CHEMS-EDDIN}
	\address{Mohamed Mahmoud CHEMS-EDDIN: Department of Mathematics, Faculty of Sciences Dhar El Mahraz, Sidi Mohamed Ben Abdellah University, Fez,  Morocco}
	\email{2m.chemseddin@gmail.com}
	
	\author[H. El Mamry]{Hamza El Mamry}
	\address{Hamza EL MAMRY: Departement of Mathematics, Faculty of Sciences Dhar El Mahraz,  Sidi Mohamed Ben Abdellah University, Fez, Morocco}
	\email{Hamza.elmamry@usmba.ac.ma}

	
	\maketitle
	\begin{abstract} 
		In this paper, we calculate the unit groups  and the $2$-class numbers of the   fields $ \KK= \mathbb{Q}(\sqrt{2}, \sqrt{q}, \sqrt{rs})$, where $q\equiv3\pmod4$, $r$ and $s$ are distinct prime numbers satisfying some congruent conditions. As an application of our main results, we derive the precise 2-class number for each field in our family. Furthermore, we show that Greenberg's conjecture is verified for several infinite families of real biquadratic fields.
	\end{abstract}
	
	\section*{Introduction}
	Determining the unit group of a number field is one of the central problems of algebraic number theory since it has many applications in particular for controlling the arithmetic of the ring of integers.  For number field $k$ of degree $n$  (i.e., $[k: \QQ] = n$), Dirichlet’s unit theorem gives the rank of the unit group, but constructing an explicit fundamental system of units is in general a difficult task, and the difficulty grows rapidly with the degree.\\ Multiquadratic fields are a particularly favorable family in this respect. Their Galois group is elementary abelian, isomorphic to $(\ZZ/2\ZZ)^n$ so the units of such a field can be related to those of its quadratic subfields through Kuroda’s class number formula and Wada’s method.  An early study of unit groups of multiquadratic fields was established by Varmon \cite{Varmon}.    For quartic bicyclic fields, Kubota \cite{Ku-56} gave
	a method for finding a fundamental system of units.
	Wada \cite{wada} generalized Kubota's method, creating an algorithm for computing
	fundamental units in any given multiquadratic field. However, in general, it is not easy to  compute the unit group of a number field  especially for number fields of degree more than $4$. Many authors tried  to characterize their Hasse unit index. For instance, Hirabayashi      and  Yoshino \cite{HY,HY2} gave criteria to determine the Hasse unit index of some multiquadratic number fields of degree $4$ and $8$. 
	In \cite{H4,H3},   Hirabayashi   gave a necessary and sufficient condition for  a multiquadratic  field of degree 16 to have a unit index equal  to $2$.  	For more works in the same direction, we refer the reader to \cite{azizunints99,azizitalbi,benjlemschi,Bulant}.\\
	In this direction,   Chems-Eddin and his collaborators obtained several results on the unit groups of real multiquadratic fields and on their connections with $2$-class groups, in particular for families of triquadratic fields of the form $\QQ(\sqrt{2},\sqrt{p},\sqrt{q})$ (cf. \cite{chemszekhniniaziziUnits1,chemszekhniniaziziUnitssaopaolo,Chemsarithhangarica,CAZ,Chemsarithhbohimica,ChemsarithMorocan,CEH2024,firstpaper}), where $p$ and $q$ are distinct prime numbers. Similarly,   El Hamam investigated another family of triquadratic number fields whose radicands are not necessarily prime, namely fields of the form \(\mathbb{Q}(\sqrt{2}, \sqrt{pq}, \sqrt{ps})\) (see \cite{Mohachilai2025, MTBhangarica, MTBric}). 

	In this paper, we study the units of  the real triquadratic field $ \KK= \mathbb{Q}(\sqrt{2}, \sqrt{q}, \sqrt{rs})$, where $q\equiv3\pmod4$, $r$ and $s$ are distincts prime numbers satisfying some conguents conditions that promote the cyclicity of the 2-class group of $\KK$.

	
	Throughout this paper we use  the following notations. Let $k$ be a number field.\label{notations}
	\begin{enumerate}[$\star$]
		\item $E_k$: the unit group of $k$,
		\item $\mathbf{C}l_2(k)$: the $2$-class group of $k$,
		\item $q(k)=(E_{k}: \prod_{i}E_{k_i})$ is the unit index of $k$, if $k$ is multiquadratic, where   $k_i$ are  the  quadratic subfields	of $k$,
		\item $h(k)$: the  class number of  $k$,
		\item $h_2(k)$: the $2$-class number of $k$,
		\item	$k^{(n)}$: the $n$th Hilbert $2$-class field of $k$,
		\item $h_2(d)$: the $2$-class number of a quadratic field $\mathbb{Q}(\sqrt{d})$,
		\item $\varepsilon_d$: the fundamental unit of a real quadratic field $\mathbb{Q}(\sqrt{d})$,
		\item $N(\varepsilon_d)$: the norm of   $\varepsilon_d$ in the extension $\mathbb{Q}(\sqrt{d})/\mathbb{Q}$,
		
		\item $\tau_i$: defined in Page \pageref{fsu preparations},
		
		\item $\left(\frac{\cdot}{\cdot}\right)$: the Legendre symbol,
		\item $\left(\frac{\cdot}{\cdot}\right)_4$: the biquadratic symbol,
		\item FSU: fundamental system of units.
		\item $[m,n]$: denotes the square root of a fundamental unit of quadratic fields when it exist,  $\sqrt{\vep_{q}}=a_1\sqrt{m}+a_2\sqrt{n}$ for some integers or semi integers $a_1$ and $a_2$,
		\item $\delta_{a,b}$:  the Kronecker delta, i.e., $\delta_{a,b}=1$ or 0  according to whether $a=b$ or not.
		\item $\rho, \gamma$:  two elements of $\{1,2\}$ such that $\rho\neq \gamma$,
		\item $k_{\infty}$: a $\ZZ_2-$ extension of k,
		\item $A(k_{\infty})$: the 2-Iwasawa module for $k_{\infty}/k$.
		
	\end{enumerate}
	\section{\bf Preliminaries } \label{sec2prep}	
	Let us start this section by  recalling the method given in    \cite{wada}, that describes a fundamental system  of units of a real  multiquadratic field $K_0$. Let  $\tau_1$ and 
	$\tau_2$ be two distinct elements of order $2$ of the Galois group of $K_0/\mathbb{Q}$. Let $K_1$, $K_2$ and $K_3$ be the three subextensions of $K_0$ invariant by  $\tau_1$,
	$\tau_2$ and $\tau_1\tau_2$, respectively. Let $\varepsilon$ denote a unit of $K_0$. Then \label{algo wada}
	$$\varepsilon^2=\varepsilon\varepsilon^{\tau_1}  \varepsilon\varepsilon^{\tau_2}(\varepsilon^{\tau_1}\varepsilon^{\tau_2})^{-1},$$
	and we have, $\varepsilon\varepsilon^{\tau_1}\in E_{K_1}$, $\varepsilon\varepsilon^{\tau_2}\in E_{K_2}$  and $\varepsilon^{\tau_1}\varepsilon^{\tau_2}\in E_{K_3}$.
	It follows that the unit group of $K_0$  
	is generated by the elements of  $E_{K_1}$, $E_{K_2}$ and $E_{K_3}$, and the square roots of elements of   $E_{K_1}E_{K_2}E_{K_3}$ which are perfect squares in $K_0$.\label{wada}
	
	While this approach is highly effective for calculating a fundamental system of units in real biquadratic fields, determining the unit group in real triquadratic fields presents a far more formidable challenge. The latter requires specialized computational techniques and careful elimination procedures using ramification theory for example, which will be examined in the following section.
	We shall consider the field $\KK=\mathbb{Q}(\sqrt{2},\sqrt{q},\sqrt{ rs})$, where $q\equiv3\pmod4$ and $r$ and $s$ are disticts prime numbers.

	We have the following diagram:

	\begin{figure}[H]
		$$\xymatrix@R=0.8cm@C=0.3cm{
			&\KK=\QQ( \sqrt 2, \sqrt{q}, \sqrt{rs})\ar@{<-}[d] \ar@{<-}[dr] \ar@{<-}[ld] \\
			k_1=\QQ(\sqrt 2,\sqrt{q})\ar@{<-}[dr]& k_2=\QQ(\sqrt 2, \sqrt{rs}) \ar@{<-}[d]& k_3=\QQ(\sqrt 2, \sqrt{qrs})\ar@{<-}[ld]\\
			&\QQ(\sqrt 2)}$$
		\caption{Intermediate fields of $\KK/\QQ(\sqrt 2)$}\label{fig:I}
	\end{figure}
	
	%

	Let $\tau_1$, $\tau_2$, $\tau_3\in\mathrm{Gal}(K_1/\QQ)$ be defined
	by
	\begin{center}
		\begin{tabular}{lll}
			$\tau_1(\sqrt{2})=-\sqrt{2}$, &
			$\tau_1(\sqrt{q})=\sqrt{q}$, &
			$\tau_1(\sqrt{rs})=\sqrt{rs}$,\\[4pt]
			$\tau_2(\sqrt{2})=\sqrt{2}$, &
			$\tau_2(\sqrt{q})=-\sqrt{q}$, &
			$\tau_2(\sqrt{rs})=\sqrt{rs}$,\\[4pt]
			$\tau_3(\sqrt{2})=\sqrt{2}$, &
			$\tau_3(\sqrt{q})=\sqrt{q}$, &
			$\tau_3(\sqrt{rs})=-\sqrt{rs}$.
		\end{tabular}
	\end{center}
	Note that  $\mathrm{Gal}(\KK/\QQ)=\langle \tau_1, \tau_2, \tau_3\rangle$
	and the subfields  $k_1$, $k_2$ and $k_3$ are
	fixed by  $\langle \tau_3\rangle$, $\langle\tau_2\rangle$ and $\langle\tau_2\tau_3\rangle$ respectively. Therefore,\label{fsu preparations} a fundamental system of units  of $\KK$ consists  of seven  units chosen from those of $k_1$, $k_2$ and $k_3$, and  from the square roots of the elements of $E_{k_1}E_{k_2}E_{k_3}$ which are squares in $\KK$. 
	
	\medskip

	First, we briefly recall 
	the class number formula for multiquadratic number fields, which is 
	traditionally associated with Kuroda \cite{Ku-50} and Wada \cite{wada}, 
	though it was already known to Herglotz \cite{He-22}
	\begin{lemma}[\cite{Ku-50}]\label{wada's f.}
		Let $k$ be a multiquadratic number field of degree $2^n$, with  $n\geq 2$ is an integer,  and $k_i$ the $s=2^n-1$ quadratic subfields of $k$. Then $h(k)$, the class number of $k$, is given by the following formula:
		$$h(k)=\frac{1}{2^v}q(k)\prod_{i=1}^{s}h(k_i),$$
		where  $ q(k)=[E_k: \prod_{i=1}^{s}E_{k_i}]$ and   $$     v=\left\{ \begin{array}{cl}
			n(2^{n-1}-1); &\text{ if } k \text{ is real, }\\
			(n-1)(2^{n-2}-1)+2^{n-1}-1 & \text{ if } k \text{ is imaginary.}
		\end{array}\right.$$
	\end{lemma}

	The following is results a particular case of the well known Fukuda's Theorem  \cite{fukuda}.

	\begin{lemma}[\cite{fukuda}]\label{lm fukuda}
		Let $k_\infty/k$ be a $\mathbb{Z}_2$-extension and $n_0$  an integer such that any prime of $k_\infty$ which is ramified in $k_\infty/k$ is totally ramified in $k_\infty/k_{n_0}$. 
		If there exists an integer $n\geq n_0$ such that   $h_2(k_n)=h_2(k_{n+1})$. Then $$h_2(k_n)=h_2(k_{m}) \;\;\;\text{for all}\;\;\;m\geq n.$$
	\end{lemma}
	
	The following lemma motivates the congruence conditions imposed on \(q\), \(r\) and \(s\) during this work, as it characterizes the cases in which our field \(\mathbb{K}\) has a cyclic \(2\)-class group.
	
	\begin{lemma}\cite[Lemma 3.5]{ChemseddinGreenbergConjectureI}\label{chemsgreen}
		Let $q\equiv3\pmod4$, r and s be distincts prime numbers. Then $\mathbf{C}l_2(\KK)$ is cyclic non trivial if and only if  one of the following conditions holds:
		\begin{enumerate}[$C1)$]
			\item \(\left(\frac{q}{s}\right) = \left(\frac{q}{r}\right) = -1\) and we have one of the following congruence conditions:\begin{enumerate}[$1)$]
				\item \(r \equiv 5 \pmod 8\) and \(s \equiv 3 \pmod 8\),
				\item \(r \equiv 3\) or \(5 \pmod 8\), \(s \equiv 7 \pmod 8\) and \(q \equiv 3 \pmod 8\).
			\end{enumerate}
			\item \(\left(\frac{q}{r}\right) = -\left(\frac{q}{s}\right) = -1\) and we have one of the following congruence conditions:
			\begin{enumerate}[$1)$]
				\item \(r \equiv 5 \pmod 8\) and \(s \equiv 5\) or \(3 \pmod 8\),
				\item \(r \equiv 3 \pmod 8\), \(s \equiv 3 \pmod 8\) and \(q \equiv 7 \pmod 8\),
				\item \(r \equiv 3 \pmod 8\) and \(s \equiv 5 \pmod 8\),
				\item \(r \equiv 7 \pmod 8\), \(s \equiv 3\) or \(5 \pmod 8\) and \(q \equiv 3 \pmod 8\),
				\item \(r \equiv 3\) or \(5 \pmod 8\), \(s \equiv 7 \pmod 8\) and \(q \equiv 3 \pmod 8\).
			\end{enumerate}	
			\item \(\left(\frac{q}{r}\right) = \left(\frac{q}{s}\right) = 1\) and we have one of the following congruence conditions:
			\begin{enumerate}[$1)$]
				\item \(r \equiv 7 \pmod 8\), \(s \equiv 3\) or \(5 \pmod 8\), and \(q \equiv 3 \pmod 8\),
				\item \(r \equiv 3 \pmod 8\) and \(s \equiv 5 \pmod 8\).
			\end{enumerate}			
		\end{enumerate}
	\end{lemma}

	%
	%
	%

	We conclude this section by stating the following lemmas, which provide explicit descriptions of the square roots of some fundamental units of quadratic number fields that will be needed later.

	\begin{lemma}\cite[Lemma 6]{BS25}\label{unitsrs}
		Let \( r, s \) be distinct primes \( \equiv 3 \mod 4 \). Then 
		\begin{enumerate}
			\item  $\sqrt{\varepsilon_{rs}} = \frac{1}{2}(a\sqrt{r} + b\sqrt{s}), \text{ for some } a, b \in \mathbb{Z} \text{ with } a \equiv b \mod 2;$   
			moreover, $ \frac{1}{4}(a^2r - b^2s) = (r/s) = -(s/r)$; 
			
			\item  $\sqrt{\varepsilon_{2r}} = a\sqrt{2} + b\sqrt{r}, \text{ for some } a, b \in \mathbb{Z};$   
			moreover,  $a^2 - b^2 = (2/r) = (-r/2);$
			\item  $\sqrt{\varepsilon_r} = \frac{1}{2}(a\sqrt{2r} + b\sqrt{2}), \text{ for some odd } a, b \in \mathbb{Z};$  
			moreover,  $\frac{1}{2}(a^2r - b^2) = -(r/2). $
		\end{enumerate}
		
	\end{lemma}
	
	\begin{lemma}[\cite{fourthpaper}, Appendix]
		Let $q,r$ and $s$ be distinct prime numbers. Then we have Tables \ref{tab7} and \ref{tab8} below.
		
		\renewcommand{\arraystretch}{1.3}
		\begin{table}[H]
			\caption{ \textbf{Forms of $\sqrt{\vep_{rs}}$ and $\sqrt{\vep_{2rs}}$ }}\label{tab7}
			\centering
			\begin{tabular}[h!]{|c|p{2cm}|c|c|c|c|c|} 
				
				\hline
				\tiny{Line number}&\textbf{\tiny{Legendre Symbole}} & \tiny{Congrunce of $r$ and $s$ $\pmod 8$} & \tiny{Result 1} & \tiny Form of $\sqrt{\vep_{rs}}$ & \tiny Result 2 & \tiny Form of $\sqrt{\vep_{2rs}}$ \\
				\hline
				1& \multirow{5}{2cm}{$\left(\dfrac{r}{s}\right)=1$} & $(5,7)$    & $2r(a+1)$ & $[\sqrt{r},\sqrt{s}]$ & $r(\alpha+1)$ & $[\sqrt{2r},\sqrt{s}]$  \\
				\cline{1-1}\cline{3-7}
				
				2& & $(5,3)$   & $2r(a+1)$ & $[\sqrt{r},\sqrt{s}]$ & $r(\alpha-1)$ & $[\sqrt{2r},\sqrt{s}]$ \\
				
				\cline{1-1}\cline{3-7}
				3& & $(3,3)$    & Lemma \ref{unitsrs} & Lemma \ref{unitsrs} & $\alpha-1$ & $[\sqrt{2},\sqrt{rs}]$ \\
				\cline{1-1}\cline{3-7}
				4& & $(3,7)$    & Lemma \ref{unitsrs} & Lemma \ref{unitsrs} & $r(\alpha+1)$ & $[\sqrt{2r},\sqrt{s}]$ \\
				\cline{1-1}\cline{3-7}
				5& & $(7,3)$    & Lemma \ref{unitsrs} & Lemma \ref{unitsrs} & $2r(\alpha+1)$ & $[\sqrt{r},\sqrt{2s}]$ \\
				\hline
				6& \multirow{3}{2cm}{$\left(\dfrac{r}{s}\right)=-1$} & $(5,3)$   & $r(a+1)$ &$[\sqrt{2r},\sqrt{2s}]$ & $2r(\alpha-1)$ & $[\sqrt{r},\sqrt{2s}]$  \\

				\cline{1-1}\cline{3-7}
				7& & $(5,7)$   & $r(a-1)$ & $[\sqrt{2r},\sqrt{2s}]$ & $2r(\alpha-1)$ & $[\sqrt{r},\sqrt{2s}]$ \\
				\cline{1-1}\cline{3-7}
				8& & $(3,3)$   & Lemma \ref{unitsrs} & Lemma \ref{unitsrs} & $\alpha-1$ & $[\sqrt{2},\sqrt{rs}]$ \\

				\hline
				
			\end{tabular}
			
		\end{table}

		\renewcommand{\arraystretch}{1.3}

		\centering
		\begin{longtable}[h!]{|c|p{2.5cm}|p{2.5cm}|c|c|c|c|} 
			\caption{\textbf{Some forms of $\sqrt{\vep_{qrs}}$ and $\sqrt{\vep_{2qrs}}$ }}\label{tab8} \\
			\hline 
			\tiny{Line number} &\tiny{$r,s,q\equiv \pmod 8$} & \tiny{$(\left(\dfrac{q}{r}\right),\left(\dfrac{q}{s}\right),\left(\dfrac{r}{s}\right))$}  & Result 1 & Form of $\sqrt{\vep_{qrs}}$ & Result 2  & Form of $\sqrt{\vep_{2qrs}}$ \\
			\hline
			1&\multirow{4}{3.5cm}{$(5,7,3)$}&  $(1,1,-1)$ & $2q(x-1)$ & $[\sqrt{q},\sqrt{rs}]$& $s(z-1)$& $[\sqrt{2s},\sqrt{qr}]$  \\
			\cline{1-1}\cline{3-7}

			2&&$(-1,-1,-1)$ & $2r(x-1)$ & $[\sqrt{r},\sqrt{qs}]$& $s(z+1)$& $[\sqrt{2s},\sqrt{qr}]$  \\ 
			\cline{1-1}\cline{3-7}

			3&&$(-1,1,-1)$ & $2r(x-1)$ & $[\sqrt{r},\sqrt{qs}]$&$s(z-1)$  & $[\sqrt{2s},\sqrt{qr}]$  \\
			\cline{1-1}\cline{3-7}

			4&& $(1,-1,-1)$& $2q(x+1)$ & $[\sqrt{q},\sqrt{rs}]$& $s(z+1)$& $[\sqrt{2s},\sqrt{qr}]$  \\

			\hline
			\hline

			5&\multirow{10}{3.5cm}{$(5,3,3)$}&  $(-1,-1,1)$   & $s(x+2)$ & $[\sqrt{s},\sqrt{qr}]$ & $q(z+1)$ & $[\sqrt{2q},\sqrt{rs}]$ \\
			\cline{1-1}\cline{3-7}

			6&&$(-1,1,1)$   & $s(x-2)$ & $[\sqrt{s},\sqrt{qr}]$ & $q(z-1)$ & $[\sqrt{2q},\sqrt{rs}]$ \\

			\cline{1-1}\cline{3-7}
			
			7&&$(1,1,-1)$   & $q(x+2)$ & $[\sqrt{q},\sqrt{rs}]$ & $s(z+1)$ & $[\sqrt{2s},\sqrt{qr}]$ \\

			\cline{1-1}\cline{3-7}
			
			8&&\multirow{3}{3cm}{$(-1,-1,-1)$}   & \multirow{3}{3cm}{$r(x-2)$}  &\multirow{3}{3cm}{$[\sqrt{r},\sqrt{qs}]$}  & $q(z+1)$ & $[\sqrt{2q},\sqrt{rs}]$ \\
			\cline{1-1}\cline{6-7}
			9&&  &  &  & $r(z+1)$ & $[\sqrt{2r},\sqrt{qs}]$ \\
			\cline{1-1}\cline{6-7}
			10&&  &  &  & $s(z-1)$ & $[\sqrt{2s},\sqrt{qr}]$ \\

			\cline{1-1}\cline{3-7}
			
			11&& \multirow{3}{3cm}{$(-1,1,-1)$}   &\multirow{3}{3cm}{$r(x-2)$}  & \multirow{3}{3cm}{$[\sqrt{r},\sqrt{qs}]$} & $q(z-1)$ & $[\sqrt{2q},\sqrt{rs}]$ \\
			\cline{1-1}\cline{6-7}
			12&& & & & $r(z+1)$ & $[\sqrt{2r},\sqrt{qs}]$ \\
			\cline{1-1}\cline{6-7}
			13&& & & & $s(z+1)$ & $[\sqrt{2s},\sqrt{qr}]$ \\
			
			\cline{1-1}\cline{3-7}
			
			14&&$(1,-1,-1)$   & $q(x-2)$ & $[\sqrt{q},\sqrt{rs}]$ & $s(z-1)$ & $[\sqrt{2s},\sqrt{qr}]$ \\

			\hline
			\hline

			15&\multirow{6}{3.5cm}{$(5,3,7)$}& $(-1,1,1)$ & $2s(x-1)$ & $[\sqrt{s},\sqrt{qr}]$& $q(z-1)$& $[\sqrt{2q},\sqrt{rs}]$  \\ 
			\cline{1-1}\cline{3-7}

			16&&$(1,1,-1)$ & $2q(x+1)$ & $[\sqrt{q},\sqrt{rs}]$& $2q(z+1)$& $[\sqrt{q},\sqrt{2rs}]$  \\
			
			\cline{1-1}\cline{3-7}

			17&&$(-1,-1,-1)$ & $2r(x-1)$ & $[\sqrt{r},\sqrt{qs}]$& $q(z+1)$ & $[\sqrt{2q},\sqrt{rs}]$ \\
			
			\cline{1-1}\cline{3-7}

			18&&$(-1,1,-1)$ & $2r(x-1)$ & $[\sqrt{r},\sqrt{qs}]$& $q(z-1)$& $[\sqrt{2q},\sqrt{rs}]$  \\
			\cline{1-1}\cline{3-7}

			19&&$(1,-1,-1)$ & $2q(x-1)$ & $[\sqrt{q},\sqrt{rs}]$& $2q(z-1)$& $[\sqrt{q},\sqrt{2rs}]$  \\
			\cline{1-1}\cline{3-7}
			
			20&& $(-1,-1,1)$  & $2s(x+1)$ & $[\sqrt{s},\sqrt{qr}]$ & $q(z+1)$& $[\sqrt{2q},\sqrt{rs}]$  \\

			\hline		\hline

			21 &$(5,5,3)$ & $(-1,1,-1)$&$s(x-1)$ & $[\sqrt{2s},\sqrt{2qr}]$ & $r(z+1)$& $[\sqrt{2r},\sqrt{qs}]$\\ 
			
			\hline		\hline
			
			22 &$(5,5,7 )$ & $(-1,1,-1)$&$s(x+1)$ &$[\sqrt{2s},\sqrt{2qr}]$ & $r(z-1)$& $[\sqrt{2r},\sqrt{qs}]$\\

			\hline 		\hline
			
			23&$(3,3,7)$ &$(-1,1,-1)$&$s(x-1)$ &$[\sqrt{2s},\sqrt{2qr}]$ &$r(z+1)$ & $[\sqrt{2r},\sqrt{qs}]$\\
			
			\hline 		\hline
			
			24&$(3,3,7)$ &$(-1,1,1)$&$2s(x-1)$ &$[\sqrt{s},\sqrt{qr}]$ &$2r(z+1)$ & $[\sqrt{r},\sqrt{2qs}]$\\
			
			\hline

		\end{longtable}

	\end{lemma}

	\section{Main results}

	%
		%
		%
		%
	%
	%
	%
		%
		%

	\begin{theorem}\label{unitstriqua}
		Let $q\equiv3\pmod{4},$
		$r$ and $s$ be distinct prime numbers.  Put
		$\vep_{2qrs}=z+t\sqrt{2qrs}$ for some integers $z$ and $t$.
		Then:
		\begin{enumerate}[\rm 1)]

			\item  If $\left(\dfrac{q}{r}\right)=-\left(\dfrac{q}{s}\right)=\left(\dfrac{r}{s}\right)=-1$ and 
			$r\equiv s\equiv5\pmod{8}$. The unit group of $\KK$ is given by:
			\begin{enumerate}[$\bullet$]
				\item If
				$\sqrt{\vep_2\vep_{rs}\vep_{2rs}}\in\QQ(\sqrt{2},\sqrt{\vep_{rs}})$, we have:
				\[
				E_{\KK}=\langle{-1,\,\vep_{2},\,\vep_{rs},\,\vep_{qrs},\,
					\sqrt{\vep_{q}},\,\sqrt{\vep_{2q}},\,
					\sqrt{\vep_{qrs}\vep_{2qrs}},\,
					\sqrt{\vep_2\vep_{rs}\vep_{2rs}}}\rangle.
				\]
				\item
				If $\sqrt{\vep_2\vep_{rs}\vep_{2rs}}\notin\QQ(\sqrt{2},\sqrt{rs})$, we have:
				\[
				E_{\KK}=\langle{-1,\,\vep_{2},\,\vep_{rs},\,\vep_{qrs},\,
					\sqrt{\vep_{q}},\,\sqrt{\vep_{2q}},\,
					\sqrt{\vep_{qrs}\vep_{2qrs}},\,\sqrt{\vep_2\vep_{rs}\vep_{2rs}\vep_{qrs}}}\rangle.
				\]

			\end{enumerate}
			In all sub-cases, $h_2(\KK)=4$.
			\item   If $\left(\dfrac{q}{r}\right)=-\left(\dfrac{q}{s}\right)=-1$,
			$r\equiv s\equiv3\pmod{8}$ and $q\equiv7\pmod{8}$. Then:
			\[
			E_{\KK}=\langle{-1,\,\vep_{2},\,\vep_{rs},\,
				\sqrt{\vep_{q}},\,\sqrt{\vep_{2q}}, ,\,\sqrt{\vep_{2rs}},\,
				\sqrt{\vep_{qrs}\vep_{2qrs}}, \sqrt{\vep_{rs}\vep_{qrs}}}\rangle.
			\]
			
			Moreover  $h_2(\KK)=2$.
			
			\item If $\left(\dfrac{r}{s}\right)=-1$,
			$q\equiv 3\pmod{8}, r\equiv 5\pmod{8}$ and $s\equiv7\pmod{8}$. Then:
			\[
			E_{\KK}=\langle{-1,\,\vep_{2},\,\vep_{rs},\, 
				\sqrt{\vep_{q}},\,\sqrt{\vep_{2q}}, \sqrt{\vep_{rs}\vep_{2rs}},  \sqrt{\vep_{rs}\vep_{2qrs}},\eta}\rangle.
			\]
			Where $\eta=\sqrt{\vep_{qrs}}$ or $\sqrt{\vep_{rs}\vep_{qrs}}$ according to whether $\left(\dfrac{q}{r}\right)=1$ or $-1.$ Moreover  $h_2(\KK)=2$.
			
			\item If $\left(\dfrac{q}{r}\right)=\left(\dfrac{r}{s}\right)=-1$,
			$q\equiv 3\pmod{8}, r\equiv 5\pmod{8}$ and $s\equiv3\pmod{8}$. Then: \begin{enumerate}[$\bullet$]
				\item If $r(z+1)$ is a square in $\NN$. Then 
				
				\[
				E_{\KK}=\langle{-1,\,\vep_{2},\, 
					\sqrt{\vep_{q}},\,\sqrt{\vep_{2q}}, \sqrt{\vep_{rs}\vep_{2rs}}, \sqrt{\vep_{qrs}\vep_{2qrs}}, \sqrt{\vep_{rs}\vep_{qrs}}, \sqrt{\vep_{rs}\vep_{2qrs}}}\rangle.
				\]
				\item Else
				\[
				E_{\KK}
				=\langle{-1,\,\varepsilon_{2},\,\varepsilon_{rs}
					,\,\varepsilon_{qrs},\,  \sqrt{\varepsilon_{q}},\,\sqrt{\varepsilon_{2q}}, \sqrt{\varepsilon_{rs}\varepsilon_{2rs}}}, \sqrt{\varepsilon_{rs}\varepsilon_{qrs}}, \eta\rangle.
				\]
				
				Where $\eta=\sqrt{\varepsilon_{rs}\varepsilon_{2qrs}}$ or $\sqrt{\varepsilon_{2qrs}}$ according to whether $s(z+1)$ or $q(z\pm1)$ is a square in $\NN$. Moreover  $h_2(\KK)=h_2(2qrs)$.

			\end{enumerate}

				%
				%
				%
			
			\item If $\left(\dfrac{q}{r}\right)=-\left(\dfrac{r}{s}\right)=(-1)^{\delta_{1,\gamma}}$,
			$q\equiv 3\pmod{8}, r\equiv 5\pmod{8}$ and $s\equiv3\pmod{8}$. Then: 
			\[
			E_{\KK}=\langle{-1,\,\vep_{2},\,\vep_{rs},\,
				\sqrt{\vep_{q}},\,\sqrt{\vep_{2q}}, \,
				\sqrt{\vep_{\rho qrs}},
				\sqrt{\vep_{rs}\vep_{2rs}}, \sqrt{\vep_{rs}\vep_{\gamma qrs}}}\rangle,
			\] 
			where $\gamma$ and $\rho$ are defined in page \pageref{notations}. Moreover  $h_2(\KK)=2$.
			\item If $\left(\dfrac{q}{r}\right)=-1$,
			$q\equiv 7\pmod{8}, r\equiv 5\pmod{8}$ and $s\equiv3\pmod{8}$. Then:
			\[
			E_{\KK}=\langle{-1,\,\vep_{2},\,\vep_{rs},\,
				\sqrt{\vep_{q}},\,\sqrt{\vep_{2q}}, \,
				\sqrt{\vep_{\rho qrs}},
				\sqrt{\vep_{rs}\vep_{2rs}}, \sqrt{\vep_{rs}\vep_{\gamma qrs}}}\rangle,
			\] 
			Where $\gamma=2$ if and only if $\left(\dfrac{q}{s}\right)=\left(\dfrac{r}{s}\right)=-1$. Moreover  $h_2(\KK)=2$.

		\end{enumerate}
	\end{theorem}
	Using the above theorem and \cite[Theorem 2.1]{fourthpaper}, we derive the following infinite families of biquadratic reel fields that satisfy Greenberg's conjecture.

	\begin{corollary}
		Let $k=\QQ(\sqrt{q},\sqrt{rs})$, where $q\equiv3\pmod{4},$
		$r$ and  $s$  are distinct prime numbers.  If one of the below   conditions $($i.e., $1)$-$4)$ below$)$ holds, then $A(k_{\infty})\simeq\ZZ/2\ZZ.$
		\begin{enumerate}[$1)$]
			\item $\left(\dfrac{q}{r}\right)=-\left(\dfrac{q}{s}\right)=-1$,
			$r\equiv s\equiv3\pmod{8}$ and $q\equiv7\pmod{8}$,
			
			\item  $\left(\dfrac{r}{s}\right)=-1$,
			$q\equiv 3\pmod{8}, r\equiv 5\pmod{8}$ and $s\equiv7\pmod{8}$,
			\item $\left(\dfrac{q}{r}\right)=-\left(\dfrac{r}{s}\right)=(-1)^{\delta_{1,\gamma}}$,
			$q\equiv 3\pmod{8}, r\equiv 5\pmod{8}$ and $s\equiv3\pmod{8}$,
			\item $\left(\dfrac{q}{r}\right)=-1$,
			$q\equiv 7\pmod{8}, r\equiv 5\pmod{8}$ and $s\equiv3\pmod{8}$.
			
		\end{enumerate} 
	\end{corollary}
	\begin{proof}
		All cases will be treated in the same way; we present the first one, and the rest are similar.  If \(\left(\dfrac{q}{r}\right)=-\left(\dfrac{q}{s}\right)=-1\),
		\(r\equiv s\equiv3\pmod{8}\) and \(q\equiv7\pmod{8}\), then according to \cite[Theorem 2.1]{fourthpaper} we have \(h_2(k)=2\), and by Theorem \ref{unitstriqua} we have \(h_2(\KK)=2\). Therefore, Lemma \ref{lm fukuda} guarantees that
		\[
		A(k_{\infty})=\varprojlim A(k_s)\simeq\ZZ/2\ZZ.
		\]
		
	\end{proof}

	\section{Proof Theorem \ref{unitstriqua}}
	Before starting  our proof, let us recall the  following useful result.
		\begin{remark}\label{calculq2q}
			\begin{enumerate}[$(1)$]
				\item Let $q\equiv3\pmod4$ be a prime number. By using Lemma \ref{unitsrs}  we have the following norm values:
				\begin{equation}
					\begin{array}{|c|c|c|c|c|c|c|}
						\hline
						\sigma & 1+\tau_2 & 1+\tau_1\tau_2 & 1+\tau_1\tau_3
						& 1+\tau_2\tau_3 & 1+\tau_1 & 1+\tau_3 \\
						\hline
						\sqrt{\varepsilon_{q}}^\sigma
						& (-1)^\alpha & (-1)^{\alpha+1} & -\varepsilon_q
						& (-1)^\alpha & -\varepsilon_q & \varepsilon_q \\
						\hline
						\sqrt{\varepsilon_{2q}}^\sigma
						& (-1)^\alpha & -\varepsilon_{2q} & (-1)^{\alpha+1}
						& (-1)^\alpha & (-1)^{\alpha+1} & \varepsilon_{2q} \\
						\hline
					\end{array}
				\end{equation}

				for some integer  $\alpha\in\{0,1\}$ such that 
				$(-1)^{\alpha}=\left(\frac{2}{q}\right)$.
				
				\item Using Wada's algorithm (cf. Page \pageref{wada}) a FSU of $k_1$ is given by $E_{k_1}=\langle-1,\vep_2,\sqrt{\vep_q}, \sqrt{\vep_{2q}}\rangle.$
				
				\item According to \cite[Remark~2.3]{AzZektaous}, if $r\equiv s \equiv 1\pmod4$ then 
				$\sqrt{\varepsilon_{2}\varepsilon_{rs}\varepsilon_{2rs}}$
				takes one of the following forms
				\[
				\left\{
				\begin{array}{lcl}
					\sqrt{\varepsilon_{2}\varepsilon_{rs}\varepsilon_{2rs}}
					&=& A+B\sqrt{2}+C\sqrt{rs}+D\sqrt{2rs},  \\[4pt]
					\sqrt{\varepsilon_{2}\varepsilon_{rs}\varepsilon_{2rs}}
					&=& E\sqrt{r}+F\sqrt{s}+G\sqrt{2r}+H\sqrt{2s}, 
				\end{array}
				\right.
				\]
				where $A,B,\cdots,H\in\QQ$.
			\end{enumerate}

		\end{remark}
	
		\begin{proof}[Proof of Theorem \ref{unitstriqua}]	Now we investigate the items of our theorem separably:
		\begin{enumerate}[\rm 1)]
			
			\item If $\left(\dfrac{q}{r}\right)=-\left(\dfrac{q}{s}\right)=\left(\dfrac{r}{s}\right)=-1$ and 
			$r\equiv s\equiv5\pmod{8}$.  Then 
			\begin{enumerate}
				\item Using \cite[Theorem 2.2]{AzZektaous} a FSU of $k_2$ is given by $E_{k_2}=\langle-1,\vep_2,\vep_{rs}, \eta\rangle,$ where $\eta=\sqrt{\vep_2\vep_{rs}\vep_{2rs}}$ or $\vep_{2rs}$ according to whether $\sqrt{\vep_2\vep_{rs}\vep_{2rs}}\in k_2$ or not.   
				\item By Table \ref{tab8} (cf. Lines 21 and 22) a FSU of $k_3$ is $E_{k_3}=\langle-1,\vep_2,\vep_{qrs}, \vep_{2qrs}\rangle.$
			\end{enumerate}
			We begin by the first case when $\sqrt{\vep_2\vep_{rs}\vep_{2rs}}\in k_2$, It follows that,  
			
			\[
			E_{k_1}E_{k_2}E_{k_3}
			=\langle{-1,\,\varepsilon_{2},\,\varepsilon_{rs},\,
				\varepsilon_{qrs},\,\varepsilon_{2qrs},\,
				\sqrt{\varepsilon_{q}},\,\sqrt{\varepsilon_{2q}}, \sqrt{\varepsilon_{2}\varepsilon_{rs}\varepsilon_{2}}}\rangle.
			\]
			Following Wada's method (cf. Page \pageref{wada}), let  $\chi$ be an element of $\KK$ which is the square root of an element of $E_{k_1}E_{k_2}E_{k_3}$. Therefore we assume 
			\begin{equation}
				\chi^2=\varepsilon_{2}^a\varepsilon_{rs}^b\varepsilon_{qrs}^c\varepsilon_{2qrs}^d\sqrt{\varepsilon_{q}}^e
				\sqrt{\varepsilon_{2q}}^f\sqrt{\varepsilon_{2}\varepsilon_{rs}\varepsilon_{2rs}}^g,
				\quad a,b,c,d,e,f,g\in\{0,1\},
			\end{equation}
			
			Remark that the question now becomes about the solvability in $\KK$ of the equation: 
			
			\begin{equation}\label{equa1}
				\chi^2=\varepsilon_{2}^a\varepsilon_{rs}^b\varepsilon_{qrs}^c\varepsilon_{2qrs}^d\sqrt{\varepsilon_{q}}^e
				\sqrt{\varepsilon_{2q}}^f\sqrt{\varepsilon_{2}\varepsilon_{rs}\varepsilon_{2rs}}^g,
			\end{equation}
			
			Assuming that this equation has solutions in $\KK$, we shall firstly use norm maps from  $\KK$ to its subextensions
			to eliminate the forms which do not occur. Notice that we use frequently table in Remark \ref{calculq2q} to extract norms of  square roots in $E_{k_3}$.

			\noindent\ding{224} Let us start by applying $N_{\KK/\QQ(\sqrt{q},\sqrt{rs})}=1+\tau_1$: 
			\[
			N_{\KK/\QQ(\sqrt{q},\sqrt{rs})}(\chi^2)
			=(-1)^a\cdot\varepsilon_{rs}^{2b}\cdot\varepsilon_{qrs}^{2c}\cdot
			(-1)^e\cdot\varepsilon_{q}^e\cdot(-1)^{(\alpha+1)f}\cdot\varepsilon_{rs}^g\cdot(-1)^{\beta_1 g},
			\]
			for some integer $\beta_1\in \{0,1\}$, and $\alpha$ is defined in Remark \ref{calculq2q}. Thus $a+e+(\alpha+1)f+\beta_1 e\equiv0\pmod2$. Since $\vep_q$ and $\vep_{rs}$ are not squares in $\QQ(\sqrt{q},\sqrt{rs})$, then  $g=e$.\\
			\noindent\ding{224} We apply now $N_{\KK/\QQ(\sqrt{q},\sqrt{2})}=1+\tau_3$:
			\[
			N_{\KK/\QQ(\sqrt{q},\sqrt{2})}(\chi^2)
			=\vep_2^{2a}\cdot(-1)^{b}\cdot\varepsilon_{q}^{e}\cdot\varepsilon_{2q}^f\cdot\varepsilon_{2}^e\cdot(-1)^{\beta_2 e},
			\]
			for some integer $\beta_2\in \{0,1\}$. Notice that $N(\vep_{rs})=-1$ (cf. \cite[Theorem 1]{kuvcera1995parity}). Since $\vep_2$ is not a square in $\QQ(\sqrt{2},\sqrt{q})$, then $e=0$. Hence  $b=0$. Therefore equality \eqref{equa1} becomes  
			
			\begin{equation}
				\chi^2=\varepsilon_{2}^a\varepsilon_{qrs}^c\varepsilon_{2qrs}^d
				\sqrt{\varepsilon_{2q}}^f,
			\end{equation}

			\noindent\ding{224} Finally we  apply  $N_{\KK/\QQ(\sqrt{2q},\sqrt{rs})}=1+\tau_1\tau_2$:
			\[
			N_{\KK/\QQ(\sqrt{2q},\sqrt{rs})}(\chi^2)
			=(-1)^{a}\cdot\varepsilon_{2qrs}^{2d}\cdot\varepsilon_{2q}^f\cdot(-1)^f,
			\]
			Since $\vep_{2q}$ is not a square in $\QQ(\sqrt{2q},\sqrt{rs})$, then $f=0$, and consequently we have $a=0$. Therefore
			\begin{equation}\label{eq2}
				\chi^2=\varepsilon_{qrs}^c\varepsilon_{2qrs}^d
			\end{equation}
			
			according to table \ref{tab8} (cf. Lines 21 and 22) $\vep_{qrs}$ and $\vep_{2qrs}$ are not squares in $\KK$ but $\sqrt{\varepsilon_{qrs}\varepsilon_{2qrs} }$ is a square. Then   equation \eqref{eq2} is solvable and admit $\sqrt{\varepsilon_{qrs}\varepsilon_{2qrs} }$ as a solution. Therefore  
			\[
			E_{\KK}=\langle{-1,\,\vep_{2},\,\vep_{rs},\,\vep_{qrs},\,
				\sqrt{\vep_{q}},\,\sqrt{\vep_{2q}},\,
				\sqrt{\vep_{qrs}\vep_{2qrs}},\,
				\sqrt{\vep_2\vep_{rs}\vep_{2rs}}}\rangle.
			\]
			
			Now If $\sqrt{\vep_2\vep_{rs}\vep_{2rs}}\notin k_2$. It follows that,  
			
			\[
			E_{k_1}E_{k_2}E_{k_3}
			=\langle{-1,\,\varepsilon_{2},\,\varepsilon_{rs},\,\varepsilon_{2rs},
				\varepsilon_{qrs},\,\varepsilon_{2qrs},\,
				\sqrt{\varepsilon_{q}},\,\sqrt{\varepsilon_{2q}}}\rangle.
			\]
			let  $\chi$ be an element of $\KK$ of the form : 
			
			\begin{equation}
				\chi^2=\varepsilon_{2}^a\varepsilon_{rs}^b\varepsilon_{2rs}^c\varepsilon_{qrs}^d\varepsilon_{2qrs}^e\sqrt{\varepsilon_{q}}^f\sqrt{\varepsilon_{2q}}^g,
				\quad a,b,c,d,e,f,g\in\{0,1\},
			\end{equation}
			\noindent\ding{224} We start by applying $N_{\KK/\QQ(\sqrt{q},\sqrt{rs})}=1+\tau_1$:
			\[
			N_{\KK/\QQ(\sqrt{q},\sqrt{rs})}(\chi^2)
			=(-1)^a\cdot\varepsilon_{rs}^{2b}\cdot(-1)^c\cdot\varepsilon_{qrs}^{2d}\cdot
			(-1)^f\cdot\varepsilon_{q}^f\cdot(-1)^{(\alpha+1)g},
			\]
			where $\alpha$ is defined in Remark \ref{calculq2q} . So as $\vep_q$ is not a square in $\QQ(\sqrt{q},\sqrt{rs}),$ then $f=0$, moreover $a+c+(\alpha+1)g\equiv0\pmod2$.

			\noindent\ding{224} Now we apply  $N_{\KK/\QQ(\sqrt{2},\sqrt{q})}=1+\tau_3$:
			\[
			N_{\KK/\QQ(\sqrt{2},\sqrt{q})}(\chi^2)=\vep_2^{2a}\cdot(-1)^{b}\cdot (-1)^{c}\cdot\varepsilon_{2q}^g,
			\]
			thus $b=c$.
			\\
			\noindent\ding{224} We apply  $N_{\KK/\QQ(\sqrt{2q},\sqrt{rs})}=1+\tau_1\tau_2$:
			\[
			N_{\KK/\QQ(\sqrt{2q},\sqrt{rs})}(\chi^2)
			=(-1)^a\cdot\vep_{rs}^{2b}\cdot(-1)^{c}\cdot          \varepsilon_{2qrs}^{2e}\cdot(-1)^g\cdot\vep_{2q}^g,
			\]
			since $\vep_{2q}$ is not a square in $\QQ(\sqrt{2q},\sqrt{rs})$ then $g=0$, and $a=b=c$. Therefore:
			$$\chi^2=\varepsilon_{2}^a\varepsilon_{rs}^a\varepsilon_{2rs}^a\varepsilon_{qrs}^d\varepsilon_{2qrs}^e,$$
			According to Remark \ref{calculq2q}  we have $\sqrt{\vep_2\vep_{rs}\vep_{2rs}}=E\sqrt{r}+F\sqrt{s}+G\sqrt{2r}+H\sqrt{2s}$  for some rational numbers $E, F, G$ and $H$. Moreover by table \ref{tab8} we have $\sqrt{\vep_{qrs}}=y_1\sqrt{2s}+y_2\sqrt{2qr}$, $\sqrt{\vep_{2qrs}}=t_1\sqrt{2r}+t_2\sqrt{qs}$ for some rational numbers $y_1, y_2, t_1$ and $t_2$, then $\sqrt{\vep_2\vep_{rs}\vep_{2rs}\vep_{qrs}}, \sqrt{\vep_{qrs}\vep_{2qrs}}\in \KK$ and 
			
			\[
			E_{\KK}=\langle{-1,\,\vep_{2},\,\vep_{rs},\,\vep_{qrs},\,
				\sqrt{\vep_{q}},\,\sqrt{\vep_{2q}},\,
				\sqrt{\vep_{qrs}\vep_{2qrs}},\,\sqrt{\vep_2\vep_{rs}\vep_{2rs}\vep_{qrs}}}\rangle.
			\]

			Notice that by using \cite[Theorem 2.4]{Azitaouspalestine} we have in this subcase:

			$$\sqrt{\vep_2\vep_{rs}\vep_{2rs}}\in k_2 \iff \left( \frac{r s}{2} \right)_4 \left( \frac{2r}{s} \right)_4 \left( \frac{2s}{r} \right)_4 = -1 $$
			
			Finally according to \cite[Corollaries 18.4, 19.7 and 19.8]{connor88},  \cite[Theorem 1 ]{kuvcera1995parity},  \cite[Propositions l and 2]{BS95}, \cite[Proposition $A_2$, p.330]{kaplan76}  and Lemma \ref{wada's f.} we have:

			\begin{align*}
				h_2(\KK) & =\frac{1}{2^9}q(\KK)h_2(2)h_2(q)h_2(rs)h_2(2q)h_2(2rs)h_2(qrs)h_2(2qrs) \\
				&=\frac{1}{2^9}2^4\cdot1\cdot1\cdot2\cdot1\cdot4\cdot4\cdot4\cdot\\
				&=4.
			\end{align*}

			\item   If $\left(\dfrac{q}{r}\right)=-\left(\dfrac{q}{s}\right)=-1$,
			$r\equiv s\equiv3\pmod{8}$, $q\equiv7\pmod{8}$, then

			\begin{enumerate}
				\item According to table \ref{tab7} (cf. Lines 3 and 8) a FSU of $k_2$ is given by $E_{k_2}=\langle-1,\vep_2,\vep_{rs}, \sqrt{\vep_{2rs}}\rangle,$    \item Similarly by table
				\ref{tab8} (Lines 23 and 24) a FSU of $k_3$ is $E_{k_3}=\langle-1,\vep_2,\vep_{qrs}, \vep_{2qrs}\rangle.$
			\end{enumerate}
			Thus,
			
			\[
			E_{k_1}E_{k_2}E_{k_3}
			=\langle{-1,\,\varepsilon_{2},\,\varepsilon_{rs},\,
				\varepsilon_{qrs},\,\varepsilon_{2qrs},\,
				\sqrt{\varepsilon_{q}},\,\sqrt{\varepsilon_{2q}}, \sqrt{\varepsilon_{2rs}}}\rangle.
			\]
			As on the first case, let  $\chi$ be an element of $\KK$ which is the square root of an element of $E_{k_1}E_{k_2}E_{k_3}$. It follows that  
			\begin{equation}
				\chi^2=\varepsilon_{2}^a\varepsilon_{rs}^b\varepsilon_{qrs}^c\varepsilon_{2qrs}^d\sqrt{\varepsilon_{q}}^e
				\sqrt{\varepsilon_{2q}}^f\sqrt{\varepsilon_{2rs}}^g,
				\quad a,b,c,d,e,f,g\in\{0,1\},
			\end{equation}
			\noindent\ding{224} Let us start by applying $N_{\KK/\QQ(\sqrt{q},\sqrt{rs})}=1+\tau_1$:
			\[
			N_{\KK/\QQ(\sqrt{q},\sqrt{rs})}(\chi^2)
			=(-1)^a\cdot\varepsilon_{rs}^{2b}\cdot\varepsilon_{qrs}^{2c}\cdot
			(-1)^e\cdot\varepsilon_{q}^e\cdot(-1)^{f}\cdot(-1)^{\alpha_1 g},
			\]
			for some integer $\alpha_1\in \{0,1\}$. Since $\vep_q$ is not a square in $\QQ(\sqrt{q},\sqrt{rs})$ we have $e=0$, moreover  $a+f+\alpha_1 g\equiv0\pmod2$.
			\\
			\noindent\ding{224} We apply now $N_{\KK/\QQ(\sqrt{2q},\sqrt{rs})}=1+\tau_1\tau_2$:
			\[
			N_{\KK/\QQ(\sqrt{2q},\sqrt{rs})}(\chi^2)
			=(-1)^{a}\cdot\vep_{rs}^{2b}\cdot\varepsilon_{2qrs}^{2d}\cdot(-1)^{f}\cdot\varepsilon_{2q}^f\cdot(-1)^{\beta_2 g},
			\]
			for some integer $\beta_2\in \{0,1\}$. As $\vep_{2q}$ is not a square in $\QQ(\sqrt{2q},\sqrt{rs})$ we get $f=0$. Therefore:

			\begin{equation}
				\chi^2=\varepsilon_{2}^a\varepsilon_{rs}^b\varepsilon_{qrs}^c\varepsilon_{2qrs}^d
				\sqrt{\varepsilon_{2rs}}^g,
			\end{equation}
			
			\noindent\ding{224} Finally we apply  $N_{\KK/\QQ(\sqrt{q},\sqrt{2rs})}=1+\tau_1\tau_3$:
			\[
			N_{\KK/\QQ(\sqrt{q},\sqrt{2rs})}(\chi^2)
			=(-1)^{a}\cdot\varepsilon_{2qrs}^{2d}\cdot(-1)^{\beta_3 g}\cdot\varepsilon_{2rs}^{g},
			\]
			for some integer $\beta_3\in \{0,1\}$. As $\vep_{2rs}$ is not a square in $\QQ(\sqrt{q},\sqrt{2rs})$ we get $g=0$, and $a=0$. Hence
			
			\begin{equation}
				\chi^2=\varepsilon_{rs}^b\varepsilon_{qrs}^c\varepsilon_{2qrs}^d,
			\end{equation}
			using tables \ref{tab7} and \ref{tab8} we conclude that $\sqrt{\vep_{qrs}\vep_{2qrs}}$ and $\sqrt{\vep_{rs}\vep_{qrs}}$ are in $\KK$, then 
			
			\[
			E_{\KK}=\langle{-1,\,\vep_{2},\,\vep_{rs},\,
				\sqrt{\vep_{q}},\,\sqrt{\vep_{2q}}, ,\,\sqrt{\vep_{2rs}},\,
				\sqrt{\vep_{qrs}\vep_{2qrs}}, \sqrt{\vep_{rs}\vep_{qrs}}}\rangle.
			\]

			\vspace{1cm}

			For the rest, we summarize some lengthy calculations and present the direct results,
			\item If $\left(\dfrac{q}{r}\right)=-\left(\dfrac{q}{s}\right)=-\left(\dfrac{r}{s}\right)=1$,
			$q\equiv 3\pmod{8}, r\equiv 5\pmod{8}$ and $s\equiv7\pmod{8}$ :
			
			According to tables \ref{tab7} et \ref{tab8} we have, 
			\begin{enumerate}
				\item A FSU of $k_2$ is given by $E_{k_2}=\langle-1,\vep_2,\vep_{rs}, \sqrt{\vep_{rs}\vep_{2rs}}\rangle,$ (cf. Table \ref{tab7} Line 7)
				\item A FSU of $k_3$ is $E_{k_2}=\langle-1,\vep_2,\vep_{qrs}, \vep_{2qrs}\rangle$ (cf. Table \ref{tab8} Line 4). 
				
			\end{enumerate}
			
			From Table \ref{tab7} Line 7 we have $2\sqrt{\vep_{rs}}=b_1\sqrt{2r}+b_2\sqrt{2s}$, $2\sqrt{\vep_{2rs}}=2\beta_1\sqrt{r}+\beta_2\sqrt{2s}$, with $b_i$ and $\beta_i$ are integers such that:  $2r\beta_1^2-s\beta_2^2=rb_1^2-sb_2^2=-2$, 
			Then

			\begin{align*}
				(4\sqrt{\vep_{qrs}\vep_{2qrs}})^{1+\tau_1} & =(2rb_1\beta_1\sqrt{2}+2b_1\beta_2\sqrt{rs}+2b_2\beta_1\sqrt{2rs}+2sb_2\beta_2)^{1+\tau_1} \\
				&=(2b_1\beta_2\sqrt{rs}+2sb_2\beta_2)^2-(2rb_1\beta_1\sqrt{2}+2b_2\beta_1\sqrt{2rs})^2\\
				&=4rb_1^2(s\beta_2^2-2r\beta_1^2)+4sb_2^2(s\beta_2^2-2r\beta_1^2)+8b_1b_2(s\beta_2^2-2r\beta_1^2)\sqrt{rs}\\
				&=8rb_1^2+8sb_2^2+16b_1b_2\sqrt{rs}\\
				&=8(rb_1^2+sb_2^2+2b_1b_2\sqrt{rs})\\
				&=16\vep_{rs},
			\end{align*}

			Similarly, we compute all the images under the norm maps and summarize these values in the following table:
			
			\begin{equation}
				\begin{array}{|c|c|c|c|c|c|c|}
					\hline
					\sigma & 1+\tau_1 & 1+\tau_2 & 1+\tau_3
					& 1+\tau_1\tau_2 & 1+\tau_1\tau_3 & 1+\tau_2\tau_3 \\
					

					\hline
					\sqrt{\varepsilon_{rs}\varepsilon_{2rs}}^\sigma
					& \vep_{rs} & \varepsilon_{rs}\vep_{2rs} & -1
					& \vep_{rs} & \vep_{2rs} & -1 \\
					\hline
				\end{array}
			\end{equation}  
			And 
			
			\[
			E_{k_1}E_{k_2}E_{k_3}
			=\langle{-1,\,\varepsilon_{2},\,\varepsilon_{rs}
				,\,\varepsilon_{qrs},\, \,\varepsilon_{qrs},\, \sqrt{\varepsilon_{q}},\,\sqrt{\varepsilon_{2q}}, \sqrt{\varepsilon_{rs}\varepsilon_{2rs}}}\rangle.
			\]
			Notice that $\sqrt{\vep_{qrs}}=[\sqrt{q},\sqrt{rs}]\in K$ (cf. Table \ref{tab8}, Line 12 ), then we search if there  is an other square by the same method as above.
			Let  $\chi$ be an element of $\KK$ which is the square root of an element of $E_{k_1}E_{k_2}E_{k_3}$. It follows that  
			\begin{equation*}
				\chi^2=\varepsilon_{2}^a\varepsilon_{rs}^b\varepsilon_{qrs}^c\varepsilon_{2qrs}^d\sqrt{\varepsilon_{q}}^e\sqrt{\varepsilon_{2q}}^f\sqrt{\varepsilon_{rs}\varepsilon_{2rs}}^g,
				\quad a,b,c,d,e,f,g\in\{0,1\}.
			\end{equation*}
			
			Therefore by applying some norm maps from $\KK$ to its subfields we get, 
			
			\begin{eqnarray}
				N_{\KK/\QQ(\sqrt{q},\sqrt{rs})}(\chi^2)&=&(-1)^a\cdot\varepsilon_{rs}^{2b}\cdot\varepsilon_{qrs}^{2c}\cdot
				(-1)^e\cdot\varepsilon_{q}^e\cdot\vep_{rs}^g,\label{5.1}\\
				N_{\KK/\QQ(\sqrt{2},\sqrt{rs})}(\chi^2)
				&=&\vep_2^{2a}\cdot\varepsilon_{rs}^{2b}\cdot
				(-1)^{e+f}\cdot(\vep_{rs}\vep_{2rs})^g,\label{5.2}\\
				N_{\KK/\QQ(\sqrt{2},\sqrt{q})}(\chi^2)
				&=&\vep_2^{2a}\cdot \vep_q^e\cdot\vep_{2q}^f\cdot(-1)^{g},\label{5.3}
			\end{eqnarray}

			hence we have by \eqref{5.3} $g=0$. Moreover  by \eqref{5.2} $e=f$, and by finally by \eqref{5.1} we have  $a=e=0$.  Therefore
			
			\begin{equation*}
				\chi^2=\varepsilon_{rs}^b\varepsilon_{qrs}^c\varepsilon_{2qrs}^d,
				\quad b,c,d\in\{0,1\}.
			\end{equation*}
			As $\sqrt{\vep_{rs}\vep_{2qrs}}, \sqrt{\vep_{qrs}}\in \KK$ and none of $\sqrt{\vep_{rs}}$, $\sqrt{\vep_{2qrs}}$ and $\sqrt{\vep_{rs}\vep_{qrs}}$ is in $\KK$. Hence :
			
			\[
			E_{\KK}=\langle{-1,\,\vep_{2},\,\vep_{rs},\, 
				\sqrt{\vep_{q}},\,\sqrt{\vep_{2q}}, \sqrt{\vep_{rs}\vep_{2rs}}, \sqrt{\vep_{qrs}}, \sqrt{\vep_{rs}\vep_{2qrs}}}\rangle.
			\]

			\item If $\left(\dfrac{q}{r}\right)=\left(\dfrac{q}{s}\right)=-\left(\dfrac{r}{s}\right)=1$,
			$q\equiv 3\pmod{8}, r\equiv 5\pmod{8}$ and $s\equiv7\pmod{8}$ 
			
			Exactly as in the previous case (5). 	According to tables \ref{tab7} et \ref{tab8} we have, 
			\begin{enumerate}
				\item A FSU of $k_2$ is given by $E_{k_2}=\langle-1,\vep_2,\vep_{rs}, \sqrt{\vep_{rs}\vep_{2rs}}\rangle$ (cf. Table \ref{tab7} Line 7),
				\item A FSU of $k_3$ is $E_{k_2}=\langle-1,\vep_2,\vep_{qrs}, \vep_{2qrs}\rangle$ (cf. Table \ref{tab8} Line 1). 
				
			\end{enumerate}
			Then 
			\begin{equation}
				\begin{array}{|c|c|c|c|c|c|c|}
					\hline
					\sigma & 1+\tau_1 & 1+\tau_2 & 1+\tau_3
					& 1+\tau_1\tau_2 & 1+\tau_1\tau_3 & 1+\tau_2\tau_3 \\
					

					\hline
					\sqrt{\varepsilon_{rs}\varepsilon_{2rs}}^\sigma
					& \vep_{rs} & \varepsilon_{rs}\vep_{2rs} & -1
					& \vep_{rs} & \vep_{2rs} & -1 \\
					\hline
				\end{array}
			\end{equation}  
			And 
			
			\[
			E_{k_1}E_{k_2}E_{k_3}
			=\langle{-1,\,\varepsilon_{2},\,\varepsilon_{rs}
				,\,\varepsilon_{qrs},\, \,\varepsilon_{2qrs},\, \sqrt{\varepsilon_{q}},\,\sqrt{\varepsilon_{2q}}, \sqrt{\varepsilon_{rs}\varepsilon_{2rs}}}\rangle.
			\]
			Notice that $\sqrt{\vep_{qrs}}=[\sqrt{q},\sqrt{rs}]\in \KK$ (cf. Table \ref{tab8}, Line 4), then we search if there  is an other square by the same method as above.
			Let  $\chi$ be an element of $\KK$ which is the square root of an element of $E_{k_1}E_{k_2}E_{k_3}$. It follows that  
			\begin{equation*}
				\chi^2=\varepsilon_{2}^a\varepsilon_{rs}^b\varepsilon_{qrs}^c\varepsilon_{2qrs}^d\sqrt{\varepsilon_{q}}^e\sqrt{\varepsilon_{2q}}^f\sqrt{\varepsilon_{rs}\varepsilon_{2rs}}^g,
				\quad a,b,c,d,e,f,g\in\{0,1\}.
			\end{equation*}
			
			Therefore by applying all norm maps from $\KK$ to its subfields we get, 
			
			\begin{eqnarray}
				N_{\KK/\QQ(\sqrt{q},\sqrt{rs})}(\chi^2)&=&(-1)^a\cdot\varepsilon_{rs}^{2b}\cdot\varepsilon_{qrs}^{2c}\cdot
				(-1)^e\cdot\varepsilon_{q}^e\cdot\vep_{rs}^g,\label{6.1}\\
				N_{\KK/\QQ(\sqrt{2},\sqrt{rs})}(\chi^2)
				&=&\vep_2^{2a}\cdot\varepsilon_{rs}^{2b}\cdot
				(-1)^{e+f}\cdot(\vep_{rs}\vep_{2rs})^g,\label{6.2}\\
				N_{\KK/\QQ(\sqrt{2},\sqrt{q})}(\chi^2)
				&=&\vep_2^{2a}\cdot \vep_q^e\cdot\vep_{2q}^f\cdot(-1)^{g},\label{6.3}
			\end{eqnarray}

			hence we have by \eqref{6.3} $g=0$. Moreover  by \eqref{6.2} $e=f$, and by finally by \eqref{6.1} we have  $a=e=0$.  Therefore
			
			\begin{equation*}
				\chi^2=\varepsilon_{rs}^b\varepsilon_{qrs}^c\varepsilon_{2qrs}^d,
				\quad b,c,d\in\{0,1\}.
			\end{equation*}
			As $\sqrt{\vep_{rs}\vep_{2qrs}}, \sqrt{\vep_{qrs}}\in \KK$ and none of $\sqrt{\vep_{rs}}$, $\sqrt{\vep_{2qrs}}$ and $\sqrt{\vep_{rs}\vep_{qrs}}$ is in $\KK$ we get :
			
			\[
			E_{\KK}=\langle{-1,\,\vep_{2},\,\vep_{rs},\, 
				\sqrt{\vep_{q}},\,\sqrt{\vep_{2q}}, \sqrt{\vep_{rs}\vep_{2rs}}, \sqrt{\vep_{qrs}}, \sqrt{\vep_{rs}\vep_{2qrs}}}\rangle.
			\]

			
			\item If $\left(\dfrac{q}{r}\right)=\left(\dfrac{q}{s}\right)=\left(\dfrac{r}{s}\right)=-1$,
			$q\equiv 3\pmod{8}, r\equiv 5\pmod{8}$ and $s\equiv7\pmod{8}$ : 
			
			According to tables \ref{tab7} et \ref{tab8} we have, 
			\begin{enumerate}
				\item A FSU of $k_2$ is given by $E_{k_2}=\langle-1,\vep_2,\vep_{rs}, \sqrt{\vep_{rs}\vep_{2rs}}\rangle$ (cf. Table \ref{tab7} Line 7)
				\item A FSU of $k_3$ is $E_{k_2}=\langle-1,\vep_2,\vep_{qrs}, \vep_{2qrs}\rangle$ (cf. Table \ref{tab8} Line 2). 
			\end{enumerate}
			Then 
			\begin{equation}
				\begin{array}{|c|c|c|c|c|c|c|}
					\hline
					\sigma & 1+\tau_1 & 1+\tau_2 & 1+\tau_3
					& 1+\tau_1\tau_2 & 1+\tau_1\tau_3 & 1+\tau_2\tau_3 \\
					

					\hline
					\sqrt{\varepsilon_{rs}\varepsilon_{2rs}}^\sigma
					& \vep_{rs} & \varepsilon_{rs}\vep_{2rs} & 1
					& \vep_{rs} & \vep_{2rs} & -1 \\
					\hline
				\end{array}
			\end{equation}  
			And 
			
			\[
			E_{k_1}E_{k_2}E_{k_3}
			=\langle{-1,\,\varepsilon_{2},\,\varepsilon_{rs}
				,\,\varepsilon_{qrs},\, \,\varepsilon_{2qrs},\, \sqrt{\varepsilon_{q}},\,\sqrt{\varepsilon_{2q}}, \sqrt{\varepsilon_{rs}\varepsilon_{2rs}}}\rangle.
			\]
			
			Let  $\chi$ be an element of $\KK$ which is the square root of an element of $E_{k_1}E_{k_2}E_{k_3}$. It follows that  
			\begin{equation*}
				\chi^2=\varepsilon_{2}^a\varepsilon_{rs}^b\varepsilon_{qrs}^c\varepsilon_{2qrs}^d\sqrt{\varepsilon_{q}}^e\sqrt{\varepsilon_{2q}}^f\sqrt{\varepsilon_{rs}\varepsilon_{2rs}}^g,
				\quad a,b,c,d,e,f,g\in\{0,1\}.
			\end{equation*}
			
			Therefore by applying all norm maps from $\KK$ to its subfields we get, 
			
			\begin{eqnarray}
				N_{\KK/\QQ(\sqrt{q},\sqrt{rs})}(\chi^2)&=&(-1)^a\cdot\varepsilon_{rs}^{2b}\cdot\varepsilon_{qrs}^{2c}\cdot
				(-1)^e\cdot\varepsilon_{q}^e\cdot\vep_{rs}^g,\label{7.1}\\
				N_{\KK/\QQ(\sqrt{2},\sqrt{rs})}(\chi^2)
				&=&\vep_2^{2a}\cdot\varepsilon_{rs}^{2b}\cdot
				(-1)^{e+f}\cdot(\vep_{rs}\vep_{2rs})^g,\label{7.2}\\
				N_{\KK/\QQ(\sqrt{2},\sqrt{q})}(\chi^2)
				&=&\vep_2^{2a}\cdot \vep_q^e\cdot\vep_{2q}^f\cdot(-1)^{g},\label{7.3}
			\end{eqnarray}

			hence we have by \eqref{7.3} $g=0$. Moreover  by \eqref{7.2} $e=f$, and by finally by \eqref{7.1} we have  $a=e=0$.  Therefore
			
			\begin{equation*}
				\chi^2=\varepsilon_{rs}^b\varepsilon_{qrs}^c\varepsilon_{2qrs}^d,
				\quad b,c,d\in\{0,1\}.
			\end{equation*}
			As $\sqrt{\vep_{rs}\vep_{2qrs}}, \sqrt{\vep_{rs}\vep_{qrs}}\in \KK$ (cf. Line 20 in Table \eqref{tab7} and Line 6 in Table \eqref{tab8}) and none of $\sqrt{\vep_{rs}}$, $\sqrt{\vep_{2qrs}}$ and $\sqrt{\vep_{qrs}}$ is in $\KK$ we get :
			
			\[
			E_{\KK}=\langle{-1,\,\vep_{2},\,\vep_{rs},\, 
				\sqrt{\vep_{q}},\,\sqrt{\vep_{2q}}, \sqrt{\vep_{rs}\vep_{2rs}}, \sqrt{\vep_{rs}\vep_{qrs}}, \sqrt{\vep_{rs}\vep_{2qrs}}}\rangle.
			\]

			
			\item If $\left(\dfrac{q}{r}\right)=-\left(\dfrac{q}{s}\right)=\left(\dfrac{r}{s}\right)=-1$,
			$q\equiv 3\pmod{8}, r\equiv 5\pmod{8}$ and $s\equiv7\pmod{8}$ :
			
			Exactly as the previous case (7)
			

			\item If $\left(\dfrac{q}{r}\right)=-\left(\dfrac{q}{s}\right)=\left(\dfrac{r}{s}\right)=-1$,
			$q\equiv 3\pmod{8}, r\equiv 5\pmod{8}$ and $s\equiv3\pmod{8}$ :
			
			According to tables \ref{tab7} et \ref{tab8} we have, 
			\begin{enumerate}
				\item A FSU of $k_2$ is given by $E_{k_2}=\langle-1,\vep_2,\vep_{rs}, \sqrt{\vep_{rs}\vep_{2rs}}\rangle$ (cf. Table \ref{tab7} Line 6),
				\item A FSU of $k_3$ is $E_{k_2}=\langle-1,\vep_2,\vep_{qrs}, \eta \rangle,$ with $\eta=\sqrt{\vep_{qrs}\vep_{2qrs}}$ or $\vep_{2qrs}$ according to whether $r(z+1)$ is a square in $\NN$ or not (cf. Table \ref{tab8} lines 11, 12 and 13).
			\end{enumerate}
			\noindent\ding{224} We begin by the first case when $r(z+1)$ is a square in $\NN$. Then we have 
			\begin{equation}
				\begin{array}{|c|c|c|c|c|c|c|}
					\hline
					\sigma & 1+\tau_1 & 1+\tau_2 & 1+\tau_3
					& 1+\tau_1\tau_2 & 1+\tau_1\tau_3 & 1+\tau_2\tau_3 \\
					\hline
					\sqrt{\varepsilon_{qrs}\vep_{2qrs}}^\sigma
					& -\vep_{qrs} & -1 & -1
					& \vep_{2qrs} & \vep_{2qrs} & \varepsilon_{qrs}\vep_{2qrs} \\
					\hline
					\sqrt{\varepsilon_{rs}\varepsilon_{2rs}}^\sigma
					& \vep_{rs} & \varepsilon_{rs}\vep_{2rs} & -1
					& \vep_{rs} & -\vep_{2rs} & -1 \\
					\hline
				\end{array}
			\end{equation}  
			
			Let  $\chi$ be an element of $\KK$ which is the square root of an element of $E_{k_1}E_{k_2}E_{k_3}$. It follows that  
			\begin{equation*}
				\chi^2=\varepsilon_{2}^a\varepsilon_{rs}^b\varepsilon_{qrs}^c\sqrt{\varepsilon_{q}}^d\sqrt{\varepsilon_{2q}}^e\sqrt{\varepsilon_{rs}\varepsilon_{2rs}}^f\sqrt{\varepsilon_{qrs}\varepsilon_{2qrs}}^g,
				\quad a,b,c,d,e,f,g\in\{0,1\}.
			\end{equation*}
			Then by applying all norm maps from $\KK$ to its subfields we get:

			\begin{eqnarray}
				N_{\KK/\QQ(\sqrt{q},\sqrt{rs})}(\chi^2)
				&=&(-1)^{a}\cdot\vep_{rs}^{2b}\cdot \vep_{qrs}^{2c}\cdot(-1)^d\cdot\vep_{q}^d\cdot\vep_{rs}^f\cdot(-1)^g\cdot\vep_{qrs}^g , \label{9.1}\\
				N_{\KK/\QQ(\sqrt{2},\sqrt{rs})}(\chi^2)
				&=&\vep_2^{2a}\cdot \vep_{rs}^{2b}\cdot(-1)^d\cdot(-1)^e\cdot(\vep_{rs}\vep_{2rs})^f\cdot(-1)^g, \label{9.2}\\
				N_{\KK/\QQ(\sqrt{2},\sqrt{q})}(\chi^2)
				&=&\vep_{2}^{2a}\cdot\varepsilon_{q}^{d}\cdot\varepsilon_{2q}^{e}\cdot
				(-1)^f\cdot
				(-1)^g, \label{9.3}\\
				N_{\KK/\QQ(\sqrt{2q},\sqrt{rs})}(\chi^2)
				&=&(-1)^{a}\cdot\varepsilon_{rs}^{2b}\cdot
				(-1)^{e}\cdot\vep_{2q}^e\cdot\vep_{rs}^f\cdot\vep_{2qrs}^g, \label{9.4}
			\end{eqnarray}

			Then the first equality \eqref{9.1} implies that $a+d+g\equiv0\pmod2$, and \eqref{9.3} implies that $f=g$. Hence again by \eqref{9.1} we get $d=0$ (since $\vep_q$ is not a square in $\QQ(\sqrt{q},\sqrt{rs})$, and by the same argument \eqref{9.4} implies that $e=0$ and $a=0$.
			Therefore we have:
			$$a=d=e=f=g=0\; \text{and} \;\; \chi^2=\varepsilon_{rs}^b\varepsilon_{qrs}^c\;\;\text{with} \; b,c\in \{0,1\}$$

			Notice that $\sqrt{\vep_{rs}},\sqrt{\vep_{qrs}}\notin \KK$ and $\sqrt{\vep_{rs}\vep_{qrs}}, \sqrt{\vep_{rs}\vep_{2qrs}}\in \KK$. Hence,

			\[
			E_{\KK}=\langle{-1,\,\vep_{2},\, 
				\sqrt{\vep_{q}},\,\sqrt{\vep_{2q}}, \sqrt{\vep_{rs}\vep_{2rs}}, \sqrt{\vep_{qrs}\vep_{2qrs}}, \sqrt{\vep_{rs}\vep_{qrs}}, \sqrt{\vep_{rs}\vep_{2qrs}}}\rangle.
			\]

			\noindent\ding{224} Now if $r(z+1)$ in \textbf{not} a square in $\NN$ (which implies that $s(z+1)$ or $q(z-1)$ is a square in $\NN$ according to Table \ref{tab8} lines 25 and 27). Then $E_{k_2}=\langle-1,\vep_2,\vep_{qrs}, \vep_{2qrs}\rangle$ and

			\[
			E_{k_1}E_{k_2}E_{k_3}
			=\langle{-1,\,\varepsilon_{2},\,\varepsilon_{rs}
				,\,\varepsilon_{qrs},\, \,\varepsilon_{2qrs},\, \sqrt{\varepsilon_{q}},\,\sqrt{\varepsilon_{2q}}, \sqrt{\varepsilon_{rs}\varepsilon_{2rs}}}\rangle.
			\]

			Let  $\chi$ be an element of $\KK$ which is the square root of an element of $E_{k_1}E_{k_2}E_{k_3}$. therefore:

			\begin{equation*}
				\chi^2=\varepsilon_{2}^a\varepsilon_{rs}^b\varepsilon_{qrs}^c\varepsilon_{2qrs}^d\sqrt{\varepsilon_{q}}^e\sqrt{\varepsilon_{2q}}^f\sqrt{\varepsilon_{rs}\varepsilon_{2rs}}^g,
				\quad a,b,c,d,e,f,g\in\{0,1\}. 
			\end{equation*}

			Then by applying some norms from $\KK$ to its subfields we get:

			\begin{eqnarray}
				N_{\KK/\QQ(\sqrt{q},\sqrt{rs})}(\chi^2)&=&(-1)^{a}\cdot\vep_{rs}^{2b}\cdot \varepsilon_{qrs}^{2c}\cdot(-1)^e\cdot\vep_{q}^e\cdot\vep_{rs}^g,\label{92.1}\\
				N_{\KK/\QQ(\sqrt{2},\sqrt{rs})}(\chi^2)
				&=&\vep_{2}^{2a}\cdot\varepsilon_{rs}^{2b}\cdot
				(-1)^{e}\cdot(-1)^f\cdot(\vep_{rs}\vep_{2rs})^g,\label{92.2}\\
				N_{\KK/\QQ(\sqrt{2},\sqrt{q})}(\chi^2)
				&=&\vep_2^{2a}\cdot \vep_{q}^{c}\cdot\vep_{2q}^f\cdot(-1)^g ,\label{92.3}
			\end{eqnarray}
			
			As $\vep_{q}$ and $\vep_{rs}$ are not squares in $\QQ(\sqrt{q},\sqrt{rs})$, then  using equality \eqref{92.1} we have  $e=g$ and $a=e$. Moreover by \eqref{92.2} we get $e=f$. Furthermore equality \eqref{92.3} implies  that $g=0$. Then

			$$a=e=f=g=0\; \text{and} \;\; \chi^2=\varepsilon_{rs}^b\varepsilon_{qrs}^c\varepsilon_{qrs}^d\;\;\text{with} \; b,c,d\in \{0,1\}.$$
			
			The above equation admits different solutions according to whether $s(z+1)$ or $q(z-1) \in \NN$. In case where $s(z+1)\in \NN$ we have, by  Table \ref{tab7} line 18 and Table \ref{tab8} line 27, $\sqrt{\vep_{rs}\vep_{qrs}}, \sqrt{\vep_{rs}\vep_{2qrs}}\in \KK$ but no one of the other squares is in $\KK$, then
			
			\[
			E_{\KK}
			=\langle{-1,\,\varepsilon_{2},\,\varepsilon_{rs}
				,\,\varepsilon_{qrs},\,  \sqrt{\varepsilon_{q}},\,\sqrt{\varepsilon_{2q}}, \sqrt{\varepsilon_{rs}\varepsilon_{2rs}}}, \sqrt{\varepsilon_{rs}\varepsilon_{qrs}}, \sqrt{\varepsilon_{rs}\varepsilon_{2qrs}}\rangle.
			\]
			Finally if $q(z-1)\in \NN$, then $\sqrt{\vep_{2qrs}}$ and $\sqrt{\vep_{rs}\vep_{qrs}}$ are in $\KK$. Hence
			
			\[
			E_{\KK}
			=\langle{-1,\,\varepsilon_{2},\,\varepsilon_{rs}
				,\,\varepsilon_{qrs},\,  \sqrt{\varepsilon_{q}},\,\sqrt{\varepsilon_{2q}}, \sqrt{\varepsilon_{rs}\varepsilon_{2rs}}}, \sqrt{\varepsilon_{rs}\varepsilon_{qrs}}, \sqrt{\varepsilon_{2qrs}}\rangle.
			\]


			\item If $\left(\dfrac{q}{r}\right)=\left(\dfrac{q}{s}\right)=\left(\dfrac{r}{s}\right)=-1$,
			$q\equiv 3\pmod{8}, r\equiv 5\pmod{8}$ and $s\equiv3\pmod{8}$ :

			Similar to previous (9), since it gives the sames FSU of $k_i$ and the same value of norms (cf. Lines 8, 9 and 10 in Table \ref{tab8}). 
			

			\item If $\left(\dfrac{q}{r}\right)=\left(\dfrac{q}{s}\right)=-\left(\dfrac{r}{s}\right)=1$,
			$q\equiv 3\pmod{8}, r\equiv 5\pmod{8}$ and $s\equiv3\pmod{8}$ :

			According to tables \ref{tab7} et \ref{tab8} we have, 
			\begin{enumerate}
				\item A FSU of $k_2$ is given by $E_{k_2}=\langle-1,\vep_2,\vep_{rs}, \sqrt{\vep_{rs}\vep_{2rs}}\rangle$ (cf. Table \ref{tab7} Line 6)
				\item A FSU of $k_2$ is $E_{k_2}=\langle-1,\vep_2,\vep_{qrs}, \vep_{2qrs}\rangle$ (cf. Table \ref{tab8} Line 7).
			\end{enumerate}
			And 
			\begin{equation}
				\begin{array}{|c|c|c|c|c|c|c|}
					\hline
					\sigma & 1+\tau_1 & 1+\tau_2 & 1+\tau_3
					& 1+\tau_1\tau_2 & 1+\tau_1\tau_3 & 1+\tau_2\tau_3 \\
					
					\hline
					\sqrt{\varepsilon_{rs}\varepsilon_{2rs}}^\sigma
					& \vep_{rs} & \varepsilon_{rs}\vep_{2rs} & -1
					& \vep_{rs} & \vep_{2rs} & -1 \\
					\hline
				\end{array}
			\end{equation}  
			
			Let  $\chi$ be an element of $\KK$ which is the square root of an element of $E_{\KK}E_{k_2}E_{k_3}$. It follows that  
			\begin{equation*}
				\chi^2=\varepsilon_{2}^a\varepsilon_{rs}^b\varepsilon_{qrs}^c\varepsilon_{2qrs}^d\sqrt{\varepsilon_{q}}^e\sqrt{\varepsilon_{2q}}^f\sqrt{\varepsilon_{rs}\varepsilon_{2rs}}^g,
				\quad a,b,c,d,e,f,g\in\{0,1\}.
			\end{equation*}
			Then by applying some  norm maps from $\KK$ to its subfields we get: 
			
			\begin{eqnarray}
				N_{\KK/\QQ(\sqrt{q},\sqrt{rs})}(\chi^2)&=&(-1)^{a}\cdot\vep_{rs}^{2b}\cdot \varepsilon_{qrs}^{2c}\cdot(-1)^e\cdot\vep_{q}^e\cdot\vep_{rs}^g,\label{11.1}\\
				N_{\KK/\QQ(\sqrt{2},\sqrt{rs})}(\chi^2)
				&=&\vep_{2}^{2a}\cdot\varepsilon_{rs}^{2b}\cdot
				(-1)^{e}\cdot(-1)^f\cdot(\vep_{rs}\vep_{2rs})^g,\label{11.2}\\
				N_{\KK/\QQ(\sqrt{2},\sqrt{q})}(\chi^2)
				&=&\vep_2^{2a}\cdot \vep_{q}^{c}\cdot\vep_{2q}^f\cdot(-1)^g ,\label{11.3}
			\end{eqnarray}
			
			As $\vep_{q}$ and $\vep_{rs}$ are not squares in $\QQ(\sqrt{q},\sqrt{rs})$, then  using equality \eqref{11.1} we have  $e=g$ and $a=e$. Moreover by \eqref{11.2} we get $e=f$. Furthermore equality \eqref{11.3} implies  that $g=0$. Then 
			
			$$a=e=f=g=0\; \text{and} \;\; \chi^2=\varepsilon_{rs}^b\varepsilon_{qrs}^c\varepsilon_{qrs}^d\;\;\text{with} \; b,c,d\in \{0,1\}.$$
			According to Table \ref{tab8} line 21 we have $\sqrt{\vep_{qrs}}, \sqrt{\vep_{rs}\vep_{2qrs}}\in \KK$ and none of $\sqrt{\vep_{2qrs}}$ and $\sqrt{\vep_{rs}}$ is in $\KK$, then 
			
			\[
			E_{\KK}
			=\langle{-1,\,\varepsilon_{2},\,\varepsilon_{rs}
				\,  \sqrt{\varepsilon_{q}},\,\sqrt{\varepsilon_{2q}}, \sqrt{\varepsilon_{rs}\varepsilon_{2rs}}}, \sqrt{\varepsilon_{rs}\varepsilon_{2qrs}}, \sqrt{\varepsilon_{qrs}}\rangle.
			\]


			\item If $\left(\dfrac{q}{r}\right)=\left(\dfrac{q}{s}\right)=-\left(\dfrac{r}{s}\right)=-1$,
			$q\equiv 3\pmod{8}, r\equiv 5\pmod{8}$ and $s\equiv3\pmod{8}$ :

			According to tables \ref{tab7} et \ref{tab8} we have, 
			\begin{enumerate}
				\item A FSU of $k_2$ is given by $E_{k_2}=\langle-1,\vep_2,\vep_{rs}, \sqrt{\vep_{rs}\vep_{2rs}}\rangle$ (cf. Table \ref{tab7} Line 2)
				\item A FSU of $k_3$ is $E_{k_2}=\langle-1,\vep_2,\vep_{qrs}, \vep_{2qrs}\rangle$ (cf. Table \ref{tab8} Line 5).
			\end{enumerate}
			And 
			\begin{equation}
				\begin{array}{|c|c|c|c|c|c|c|}
					\hline
					\sigma & 1+\tau_1 & 1+\tau_2 & 1+\tau_3
					& 1+\tau_1\tau_2 & 1+\tau_1\tau_3 & 1+\tau_2\tau_3 \\
					
					\hline
					\sqrt{\varepsilon_{rs}\varepsilon_{2rs}}^\sigma
					& \vep_{rs} & \varepsilon_{rs}\vep_{2rs} & -1
					& \vep_{rs} & -\vep_{2rs} & -1 \\
					\hline
				\end{array}
			\end{equation}  
			
			Let  $\chi$ be an element of $\KK$ which is the square root of an element of $E_{\KK}E_{k_2}E_{k_3}$. It follows that  
			\begin{equation*}
				\chi^2=\varepsilon_{2}^a\varepsilon_{rs}^b\varepsilon_{qrs}^c\varepsilon_{2qrs}^d\sqrt{\varepsilon_{q}}^e\sqrt{\varepsilon_{2q}}^f\sqrt{\varepsilon_{rs}\varepsilon_{2rs}}^g,
				\quad a,b,c,d,e,f,g\in\{0,1\}.
			\end{equation*}
			Then by applying some norm maps from $\KK$ to its subfields we extract the following informations: 
			\begin{equation*}
				a=e=f=g=0	\quad \text{and} \quad	\chi^2=\varepsilon_{rs}^b\varepsilon_{qrs}^c\varepsilon_{2qrs}^d,
				\quad 				\quad b,c,d\in\{0,1\}.
			\end{equation*}
			Notice that $\sqrt{\vep_{2qrs}}, \sqrt{\vep_{rs}\vep_{qrs}}\in \KK$ (cf. Table \ref{tab8} line 16) and $\sqrt{\vep_{rs}}, \sqrt{\vep_{qrs}}\notin \KK$.  Therefore

			\[
			E_{\KK}
			=\langle{-1,\,\varepsilon_{2},\,\varepsilon_{rs}
				,\,\varepsilon_{qrs},\,  \sqrt{\varepsilon_{q}},\,\sqrt{\varepsilon_{2q}}, \sqrt{\varepsilon_{rs}\varepsilon_{2rs}}}, \sqrt{\varepsilon_{rs}\varepsilon_{qrs}}, \sqrt{\varepsilon_{2qrs}}\rangle.
			\]


			\item If $\left(\dfrac{q}{r}\right)=-\left(\dfrac{q}{s}\right)=-\left(\dfrac{r}{s}\right)=-1$,
			$q\equiv 3\pmod{8}, r\equiv 5\pmod{8}$ and $s\equiv3\pmod{8}$ :
			
			According to tables \ref{tab7} et \ref{tab8} we have, 
			\begin{enumerate}
				\item A FSU of $k_2$ is given by $E_{k_2}=\langle-1,\vep_2,\vep_{rs}, \sqrt{\vep_{rs}\vep_{2rs}}\rangle$ (cf. Table \ref{tab7} Line 2).
				\item A FSU of $k_3$ is $E_{k_2}=\langle-1,\vep_2,\vep_{qrs}, \vep_{2qrs}\rangle$  (cf. Table \ref{tab8} Line 6).
			\end{enumerate}
			And 
			\begin{equation}
				\begin{array}{|c|c|c|c|c|c|c|}
					\hline
					\sigma & 1+\tau_1 & 1+\tau_2 & 1+\tau_3
					& 1+\tau_1\tau_2 & 1+\tau_1\tau_3 & 1+\tau_2\tau_3 \\
					
					\hline
					\sqrt{\varepsilon_{rs}\varepsilon_{2rs}}^\sigma
					& \vep_{rs} & \varepsilon_{rs}\vep_{2rs} & -1
					& \vep_{rs} & -\vep_{2rs} & -1 \\
					\hline
				\end{array}
			\end{equation}  
			Let  $\chi$ be an element of $\KK$ which is the square root of an element of $E_{\KK}E_{k_2}E_{k_3}$. It follows that  
			\begin{equation*}
				\chi^2=\varepsilon_{2}^a\varepsilon_{rs}^b\varepsilon_{qrs}^c\varepsilon_{2qrs}^d\sqrt{\varepsilon_{q}}^e\sqrt{\varepsilon_{2q}}^f\sqrt{\varepsilon_{rs}\varepsilon_{2rs}}^g,
				\quad a,b,c,d,e,f,g\in\{0,1\}.
			\end{equation*}
			Then by applying some norm maps from $\KK$ to its subfields we extract the following informations: 
			\begin{equation*}
				a=e=f=g=0	\quad \text{and} \quad	\chi^2=\varepsilon_{rs}^b\varepsilon_{qrs}^c\varepsilon_{2qrs}^d,
				\quad 				\quad b,c,d\in\{0,1\}.
			\end{equation*}
			Notice that $\sqrt{\vep_{2qrs}}, \sqrt{\vep_{rs}\vep_{qrs}}\in \KK$ (cf. Table \ref{tab8} line 17) and $\sqrt{\vep_{rs}}, \sqrt{\vep_{qrs}}\notin \KK$.  Therefore

			\[
			E_{\KK}
			=\langle{-1,\,\varepsilon_{2},\,\varepsilon_{rs}
				,\,\varepsilon_{qrs},\,  \sqrt{\varepsilon_{q}},\,\sqrt{\varepsilon_{2q}}, \sqrt{\varepsilon_{rs}\varepsilon_{2rs}}}, \sqrt{\varepsilon_{rs}\varepsilon_{qrs}}, \sqrt{\varepsilon_{2qrs}}\rangle.
			\]


			\item If $\left(\dfrac{q}{r}\right)=-\left(\dfrac{q}{s}\right)=-\left(\dfrac{r}{s}\right)=1$,
			$q\equiv 3\pmod{8}, r\equiv 5\pmod{8}$ and $s\equiv3\pmod{8}$ :
			
			According to tables \ref{tab7} et \ref{tab8} we have, 
			\begin{enumerate}
				\item A FSU of $k_2$ is given by $E_{k_2}=\langle-1,\vep_2,\vep_{rs}, \sqrt{\vep_{rs}\vep_{2rs}}\rangle$ (cf. Table \ref{tab7} Line 6),
				\item A FSU of $k_2$ is $E_{k_2}=\langle-1,\vep_2,\vep_{qrs}, \vep_{2qrs}\rangle$ (cf. Table \ref{tab8} Line 14).
			\end{enumerate}
			And 
			\begin{equation}
				\begin{array}{|c|c|c|c|c|c|c|}
					\hline
					\sigma & 1+\tau_1 & 1+\tau_2 & 1+\tau_3
					& 1+\tau_1\tau_2 & 1+\tau_1\tau_3 & 1+\tau_2\tau_3 \\
					
					\hline
					\sqrt{\varepsilon_{rs}\varepsilon_{2rs}}^\sigma
					& \vep_{rs} & \varepsilon_{rs}\vep_{2rs} & -1
					& \vep_{rs} & \vep_{2rs} & -1 \\
					\hline
				\end{array}
			\end{equation}  
			
			Let  $\chi$ be an element of $\KK$ which is the square root of an element of $E_{\KK}E_{k_2}E_{k_3}$. It follows that  
			\begin{equation*}
				\chi^2=\varepsilon_{2}^a\varepsilon_{rs}^b\varepsilon_{qrs}^c\varepsilon_{2qrs}^d\sqrt{\varepsilon_{q}}^e\sqrt{\varepsilon_{2q}}^f\sqrt{\varepsilon_{rs}\varepsilon_{2rs}}^g,
				\quad a,b,c,d,e,f,g\in\{0,1\}.
			\end{equation*}
			Then by applying some  norm maps from $\KK$ to its subfields we get: 
			\begin{equation*}
				a=e=f=g=0	\quad \text{and} \quad	\chi^2=\varepsilon_{rs}^b\varepsilon_{qrs}^c\varepsilon_{2qrs}^d,
				\quad 				\quad b,c,d\in\{0,1\}.
			\end{equation*}
			As $\sqrt{\vep_{qrs}}, \sqrt{\vep_{rs}\vep_{2qrs}}\in \KK$ (cf. Table \ref{tab8} line 14) and $\sqrt{\vep_{rs}}, \sqrt{\vep_{qrs}}\notin \KK$.  Therefore

			\[
			E_{\KK}
			=\langle{-1,\,\varepsilon_{2},\,\varepsilon_{rs}
				,\,\varepsilon_{qrs},\,  \sqrt{\varepsilon_{q}},\,\sqrt{\varepsilon_{2q}}, \sqrt{\varepsilon_{rs}\varepsilon_{2rs}}}, \sqrt{\varepsilon_{rs}\varepsilon_{2qrs}}, \sqrt{\varepsilon_{qrs}}\rangle.
			\]


			\item If $\left(\dfrac{q}{r}\right)=\left(\dfrac{q}{s}\right)=\left(\dfrac{r}{s}\right)=-1$,
			$q\equiv 7\pmod{8}, r\equiv 5\pmod{8}$ and $s\equiv3\pmod{8}$ :

			Notice that $L/\KK$ is unramified quadratic extension, where $L=\QQ(\sqrt{2},\sqrt{q},\sqrt{r},\sqrt{s})$. In adition to $A(L)$ is trivial (cf. \cite[Theorem ]{paritymohib}) and $A(\KK)$  is cyclic (cf. Lemma  \ref{chemsgreen}), then $h(\KK)=2$ wich is equivalent to $q(\KK)=2^5$. \\
			we have $\sqrt{\vep_q}, \sqrt{\vep_{2q}}, \sqrt{\vep_{rs}\vep_{2rs}}, \sqrt{\vep_{2qrs}}$ (cf. Table \ref{tab8} Line 17) and $\sqrt{\vep_{rs}\vep_{qrs}}$ belongs to $\KK$, moreover they multiplicatively independent.  Thus

			\[
			E_{\KK}=\langle{-1,\,\vep_{2},\,\vep_{rs},\,
				\sqrt{\vep_{q}},\,\sqrt{\vep_{2q}}, \,
				\sqrt{\vep_{qrs}},
				\sqrt{\vep_{rs}\vep_{2rs}}, \sqrt{\vep_{rs}\vep_{2qrs}}}\rangle,
			\]


			\item If $\left(\dfrac{q}{r}\right)=\left(\dfrac{q}{s}\right)=-\left(\dfrac{r}{s}\right)=-1$,
			$q\equiv 7\pmod{8}, r\equiv 5\pmod{8}$ and $s\equiv3\pmod{8}$ :
			Notice that $L/\KK$ is unramified quadratic extension, where $L=\QQ(\sqrt{2},\sqrt{q},\sqrt{r},\sqrt{s})$. In adition to $A(L)$ is trivial (cf. \cite[Theorem ]{paritymohib}) and $A(\KK)$  is cyclic (cf. Lemma  \ref{chemsgreen}), then $h(\KK)=2$ wich is equivalent to $q(\KK)=2^5$. \\
			we have $\sqrt{\vep_q}, \sqrt{\vep_{2q}}, \sqrt{\vep_{rs}\vep_{2rs}}, \sqrt{\vep_{2qrs}}$ (cf. Table \ref{tab8} Line 20) and $\sqrt{\vep_{rs}\vep_{qrs}}$ belongs to $\KK$, moreover they multiplicatively independent.  Thus

			\[
			E_{\KK}=\langle{-1,\,\vep_{2},\,\vep_{rs},\,
				\sqrt{\vep_{q}},\,\sqrt{\vep_{2q}}, \,
				\sqrt{\vep_{2qrs}},
				\sqrt{\vep_{rs}\vep_{2rs}}, \sqrt{\vep_{rs}\vep_{qrs}}}\rangle,
			\]


			\item If $\left(\dfrac{q}{r}\right)=-\left(\dfrac{q}{s}\right)=\left(\dfrac{r}{s}\right)=-1$,
			$q\equiv 7\pmod{8}, r\equiv 5\pmod{8}$ and $s\equiv3\pmod{8}$ :

			Notice that $L/\KK$ is unramified quadratic extension, where $L=\QQ(\sqrt{2},\sqrt{q},\sqrt{r},\sqrt{s})$. In adition to $A(L)$ is trivial (cf. \cite[Theorem ]{paritymohib}) and $A(\KK)$  is cyclic (cf. Lemma  \ref{chemsgreen}), then $h(\KK)=2$ wich is equivalent to $q(\KK)=2^5$. \\
			we have $\sqrt{\vep_q}, \sqrt{\vep_{2q}}, \sqrt{\vep_{rs}\vep_{2rs}}, \sqrt{\vep_{2qrs}}$ (cf. Table \ref{tab8} Line 18) and $\sqrt{\vep_{rs}\vep_{qrs}}$ belongs to $\KK$. Thus

			\[
			E_{\KK}=\langle{-1,\,\vep_{2},\,\vep_{rs},\,
				\sqrt{\vep_{q}},\,\sqrt{\vep_{2q}}, \,
				\sqrt{\vep_{2qrs}},
				\sqrt{\vep_{rs}\vep_{2rs}}, \sqrt{\vep_{rs}\vep_{qrs}}}\rangle,
			\]


			\item If $\left(\dfrac{q}{r}\right)=-\left(\dfrac{q}{s}\right)=-\left(\dfrac{r}{s}\right)=-1$,
			$q\equiv 7\pmod{8}, r\equiv 5\pmod{8}$ and $s\equiv3\pmod{8}$ :
			Notice that $L/\KK$ is unramified quadratic extension, where $L=\QQ(\sqrt{2},\sqrt{q},\sqrt{r},\sqrt{s})$. In adition to $A(L)$ is trivial (cf. \cite[Theorem ]{paritymohib}) and $A(\KK)$  is cyclic (cf. Lemma  \ref{chemsgreen}), then $h(\KK)=2$ wich is equivalent to $q(\KK)=2^5$. \\
			we have $\sqrt{\vep_q}, \sqrt{\vep_{2q}}, \sqrt{\vep_{rs}\vep_{2rs}}, \sqrt{\vep_{2qrs}}$ (cf. Table \ref{tab8} Line 15) and $\sqrt{\vep_{rs}\vep_{qrs}}$ belongs to $\KK$. Thus

			\[
			E_{\KK}=\langle{-1,\,\vep_{2},\,\vep_{rs},\,
				\sqrt{\vep_{q}},\,\sqrt{\vep_{2q}}, \,
				\sqrt{\vep_{2qrs}},
				\sqrt{\vep_{rs}\vep_{2rs}}, \sqrt{\vep_{rs}\vep_{qrs}}}\rangle,
			\]

		\end{enumerate}
	\end{proof}

\end{document}